\documentclass[11pt,reqno]{amsart}

\usepackage[T1]{fontenc}
\usepackage[utf8]{inputenc}
\usepackage{amsmath,amssymb,amsthm}
\usepackage{mathtools}
\usepackage[hidelinks]{hyperref}

\theoremstyle{plain}
\newtheorem{theorem}{Theorem}
\newtheorem{lemma}[theorem]{Lemma}
\newtheorem{proposition}[theorem]{Proposition}
\newtheorem{corollary}[theorem]{Corollary}

\theoremstyle{definition}
\newtheorem{definition}[theorem]{Definition}
\newtheorem{question}{Question}

\theoremstyle{remark}
\newtheorem{remark}[theorem]{Remark}

\DeclareMathOperator{\TR}{TR}
\newcommand{\R}{\mathrm{R}}

\begin{document}

\title{Connected Counterexamples for Target Ramsey Numbers}

\author{Guillaume Lecomte}
\address{Independent researcher, Paris, France}
\email{guillaume.lecomteexed@edu.executive.em-lyon.com}
\urladdr{https://orcid.org/0009-0003-3532-3379}

\date{5 August 2026}


\begin{abstract}
Chartrand and Zhang asked whether there exists a graph $G$ without isolated vertices whose target Ramsey number satisfies $\TR(G) > \R(G)$. We answer the question affirmatively, even under strong structural restrictions. If $\Gamma_t = P_t \square P_t$ is the square $t \times t$ grid, then the elementary counting bound of Chartrand and Zhang gives $\TR(\Gamma_t) \geq 2t(t-1)+1$, whereas a theorem of Mota, S\'ark\"ozy, Schacht and Taraz gives $\R(\Gamma_t) = (3/2 + o(1))t^2$. Consequently, $\TR(\Gamma_t) > \R(\Gamma_t)$ for all sufficiently large $t$, so there are infinitely many connected, planar, bipartite counterexamples of maximum degree four. Higher-dimensional grids show that the ratio $\TR(G)/\R(G)$ is unbounded on connected bipartite graphs, while each fixed-dimensional witnessing family has bounded maximum degree. We also record an independent construction of disconnected counterexamples from the theorem of Burr, Erd\H{o}s and Spencer on Ramsey numbers of multiple copies: if a fixed graph $H$ satisfies $e(H) > 2v(H) - \alpha(H)$, then $\TR(qH) > \R(qH)$ for every sufficiently large $q$.
\end{abstract}

\maketitle

\section{Introduction}

For a graph $G$ without isolated vertices, the diagonal Ramsey number $\R(G)$ is the least integer $n$ such that every red--blue colouring of $E(K_n)$ contains a monochromatic copy of $G$. We write $v(G)$ and $e(G)$ for the order and size of $G$, and $\alpha(G)$ for its independence number.

Let $m = e(G)$. A \emph{Ramsey chain} with target graph $G$ is a sequence
\[
G_1, G_2, \ldots, G_m
\]
of pairwise edge-disjoint monochromatic subgraphs of a red--blue coloured complete graph such that
\[
e(G_i) = i, \qquad G_i \text{ is isomorphic to a subgraph of } G_{i+1} \quad (1 \leq i < m),
\]
and $G_m \cong G$. The \emph{target Ramsey number} $\TR(G)$ is the least integer $n$ such that every red--blue colouring of $K_n$ contains such a chain. Chartrand and Zhang proved that $\TR(G)$ exists for every graph without isolated vertices and observed that $\TR(G) \geq \R(G)$, since the last link of a chain is itself a monochromatic copy of $G$. They asked the following; see \cite{CZ1,CZ2}.

\begin{question}
Does there exist a graph $G$ without isolated vertices such that $\TR(G) > \R(G)$?
\end{question}

We give two independent affirmative answers. The first uses square grids and produces connected planar bipartite counterexamples of bounded degree. Higher-dimensional grids give arbitrarily large values of the ratio $\TR(G)/\R(G)$, although the degree bound necessarily grows with the dimension in this construction. The second combines the same elementary counting obstruction with the eventual linear behaviour of Ramsey numbers of many disjoint copies of a fixed graph, and produces disconnected counterexamples.

\section{The counting obstruction}

The following elementary observation is due to Chartrand and Zhang \cite{CZ1,CZ2}; we include the proof for completeness. It is the sole target-Ramsey input in both constructions.

\begin{lemma}\label{lem:count}
For every graph $G$ without isolated vertices,
\[
\TR(G) \geq e(G) + 1.
\]
\end{lemma}

\begin{proof}
Set $m = e(G)$, and suppose that a Ramsey chain with target $G$ occurs in $K_n$. Since its links are pairwise edge-disjoint and have sizes $1, 2, \ldots, m$, their union contains exactly
\[
\sum_{i=1}^{m} i = \binom{m+1}{2}
\]
distinct edges. Hence
\[
\binom{n}{2} \geq \binom{m+1}{2}.
\]
The map $s \mapsto \binom{s}{2}$ is strictly increasing on the positive integers, so $n \geq m+1$. Equivalently, for $n \leq m$ no colouring of $K_n$ contains a Ramsey chain with target $G$, and the conclusion follows from the definition of $\TR(G)$.
\end{proof}

\begin{proposition}\label{prop:criterion}
If a graph $G$ without isolated vertices satisfies
\[
\R(G) \leq e(G),
\]
then $\TR(G) > \R(G)$.
\end{proposition}

\begin{proof}
Lemma~\ref{lem:count} gives $\TR(G) \geq e(G) + 1 > \R(G)$.
\end{proof}

Thus it suffices to find graphs whose ordinary Ramsey number is at most their number of edges. Large grids provide a particularly structured family of this kind.

\section{Connected planar counterexamples}

Let $P_t$ denote the path on $t$ vertices and let
\[
\Gamma_t = P_t \square P_t
\]
be the square $t \times t$ grid. For $t \geq 3$, the graph $\Gamma_t$ is connected, planar, bipartite and has maximum degree four. Moreover,
\begin{equation}\label{eq:grid}
v(\Gamma_t) = t^2, \qquad e(\Gamma_t) = 2t(t-1) = 2t^2 - 2t.
\end{equation}

Mota, S\'ark\"ozy, Schacht and Taraz proved that the two-colour Ramsey number of the grid $G_{a,b} = P_a \square P_b$ satisfies
\[
\R(G_{a,b}) = \left( \tfrac{3}{2} + o(1) \right) ab,
\]
where the $o(1)$ tends to zero as $ab \to \infty$; see \cite[Corollary 1.4]{MSST}. In particular, as $t \to \infty$,
\begin{equation}\label{eq:ramseygrid}
\R(\Gamma_t) = \left( \tfrac{3}{2} + o(1) \right) t^2.
\end{equation}
Only the upper bound implicit in \eqref{eq:ramseygrid} will be used.

\begin{theorem}\label{thm:planar}
For every sufficiently large integer $t$,
\[
\TR(\Gamma_t) > \R(\Gamma_t).
\]
Consequently, there are infinitely many connected planar bipartite graphs of maximum degree four whose target Ramsey number is strictly larger than their ordinary Ramsey number.
\end{theorem}

\begin{proof}
By Lemma~\ref{lem:count} and \eqref{eq:grid},
\begin{equation}\label{eq:trgrid}
\TR(\Gamma_t) \geq 2t^2 - 2t + 1.
\end{equation}
On the other hand, \eqref{eq:ramseygrid} implies, for example, that
\[
\R(\Gamma_t) \leq \tfrac{7}{4} t^2
\]
for all sufficiently large $t$. Since
\[
2t^2 - 2t + 1 - \tfrac{7}{4} t^2 = \tfrac{1}{4} t^2 - 2t + 1 \longrightarrow +\infty,
\]
we have
\[
\R(\Gamma_t) < 2t^2 - 2t + 1 \leq \TR(\Gamma_t)
\]
for every sufficiently large $t$.
\end{proof}

\begin{corollary}
As $t \to \infty$,
\[
\TR(\Gamma_t) - \R(\Gamma_t) \geq \left( \tfrac{1}{2} - o(1) \right) t^2,
\]
and
\[
\liminf_{t \to \infty} \frac{\TR(\Gamma_t)}{\R(\Gamma_t)} \geq \frac{4}{3}.
\]
\end{corollary}

\begin{proof}
Combine \eqref{eq:trgrid} with \eqref{eq:ramseygrid}, and divide by $t^2$ for the ratio statement.
\end{proof}

\begin{remark}[No explicit first grid]
The asymptotic theorem quoted in \eqref{eq:ramseygrid} does not by itself identify the least $t$ for which $\Gamma_t$ is a counterexample. Determining the smallest connected counterexample, or even the first square grid satisfying $\TR(\Gamma_t) > \R(\Gamma_t)$, is a separate finite problem.
\end{remark}

\section{Higher-dimensional grids}

For integers $d \geq 2$ and $t \geq 2$, let
\[
\Gamma_{d,t} = P_t^{\square d}
\]
be the $d$-dimensional grid. It is connected and bipartite, its two vertex classes differ in size by at most one, and
\begin{equation}\label{eq:dgrid}
v(\Gamma_{d,t}) = t^d, \qquad e(\Gamma_{d,t}) = d\, t^{d-1}(t-1), \qquad \Delta(\Gamma_{d,t}) = 2d.
\end{equation}
Its bandwidth is at most $t^{d-1} = o(t^d)$ for fixed $d$. Hence, for fixed $d$, the two-colour result of Mota, S\'ark\"ozy, Schacht and Taraz applies and gives
\begin{equation}\label{eq:ramseydgrid}
\R(\Gamma_{d,t}) \leq \left( \tfrac{3}{2} + o(1) \right) t^d \qquad (t \to \infty,\ d \text{ fixed});
\end{equation}
see \cite[Theorem 1.2 and the discussion of higher-dimensional grids]{MSST}.

\begin{theorem}\label{thm:dgrid}
For every fixed integer $d \geq 2$ and every sufficiently large $t$,
\[
\TR(\Gamma_{d,t}) > \R(\Gamma_{d,t}),
\]
and
\[
\liminf_{t \to \infty} \frac{\TR(\Gamma_{d,t})}{\R(\Gamma_{d,t})} \geq \frac{2d}{3}.
\]
Consequently,
\[
\sup \left\{ \frac{\TR(G)}{\R(G)} : G \text{ is connected and bipartite} \right\} = \infty.
\]
More precisely, for every $M > 0$ there exists a degree bound $\Delta_M$ and an infinite family of connected bipartite graphs of maximum degree at most $\Delta_M$ for which $\liminf \TR/\R \geq M$.
\end{theorem}

\begin{proof}
By Lemma~\ref{lem:count} and \eqref{eq:dgrid},
\[
\TR(\Gamma_{d,t}) \geq d\, t^{d-1}(t-1) + 1 = (d + o(1)) t^d.
\]
Combining this with \eqref{eq:ramseydgrid} gives the asserted ratio. Since $2d/3 > 1$ for every $d \geq 2$, the strict inequality follows for all sufficiently large $t$. Given $M > 0$, choose $d$ with $2d/3 \geq M$ and take $\Delta_M = 2d$.
\end{proof}

\begin{remark}
Planarity is lost for $d \geq 3$. Thus Theorem~\ref{thm:planar} gives the stronger structural conclusion, while Theorem~\ref{thm:dgrid} gives the stronger quantitative separation.
\end{remark}

\section{A second construction from multiple copies}

For a positive integer $q$ and a graph $H$, write $qH$ for the disjoint union of $q$ copies of $H$. We use the following consequence of a theorem of Burr, Erd\H{o}s and Spencer \cite{BES}; their theorem in fact determines the exact eventual linear form.

\begin{theorem}[Burr--Erd\H{o}s--Spencer]\label{thm:bes}
For every fixed graph $H$ without isolated vertices, there exist an integer $q_0(H)$ and a constant $C_H$ such that
\[
\R(qH) \leq \bigl( 2v(H) - \alpha(H) \bigr) q + C_H
\]
for every integer $q \geq q_0(H)$.
\end{theorem}

\begin{theorem}\label{thm:copies}
Let $H$ be a fixed graph without isolated vertices. If
\begin{equation}\label{eq:crit}
e(H) > 2v(H) - \alpha(H),
\end{equation}
then
\[
\TR(qH) > \R(qH)
\]
for every sufficiently large integer $q$.
\end{theorem}

\begin{proof}
Put
\[
a_H = 2v(H) - \alpha(H), \qquad \delta_H = e(H) - a_H > 0.
\]
Since $e(qH) = q\,e(H)$, Lemma~\ref{lem:count} gives $\TR(qH) \geq q\,e(H) + 1$. By Theorem~\ref{thm:bes}, for all sufficiently large $q$ we have $\R(qH) \leq a_H q + C_H$. Consequently,
\[
\TR(qH) - \R(qH) \geq q\,e(H) + 1 - (a_H q + C_H) = \delta_H q + 1 - C_H,
\]
which is positive for all sufficiently large $q$.
\end{proof}

\begin{corollary}
For every integer $r \geq 5$ and every sufficiently large integer $q$,
\[
\TR(qK_r) > \R(qK_r).
\]
\end{corollary}

\begin{proof}
For $H = K_r$, condition \eqref{eq:crit} becomes
\[
\binom{r}{2} > 2r - 1,
\]
which holds precisely for integers $r \geq 5$.
\end{proof}

Write $K_5^-$ for the graph obtained from $K_5$ by deleting one edge. Then
\[
v(K_5^-) = 5, \qquad e(K_5^-) = 9, \qquad \alpha(K_5^-) = 2,
\]
so $2v(K_5^-) - \alpha(K_5^-) = 8 < 9 = e(K_5^-)$ and Theorem~\ref{thm:copies} applies to $K_5^-$ as well.

The next proposition records that these are the only witnesses to the criterion on at most five vertices.

\begin{proposition}
No graph without isolated vertices on at most four vertices satisfies \eqref{eq:crit}, and $K_5$ and $K_5^-$ are the only such graphs on five vertices.
\end{proposition}

\begin{proof}
Let $H$ have $v = v(H)$ vertices and independence number $\alpha = \alpha(H)$. If $A$ is an independent set of size $\alpha$, then every edge of $H$ meets $V(H) \setminus A$, and hence
\begin{equation}\label{eq:edgebound}
e(H) \leq \binom{v - \alpha}{2} + \alpha(v - \alpha).
\end{equation}
If $v \leq 4$, direct substitution for $1 \leq \alpha \leq v-1$ in \eqref{eq:edgebound} gives $e(H) \leq 2v - \alpha$; for $\alpha = 1$ this is just $\binom{v}{2} \leq 2v - 1$. Thus no such graph satisfies \eqref{eq:crit}.

Now let $v = 5$. If $\alpha = 1$, then $H = K_5$, which satisfies $10 > 9$. If $\alpha = 2$, the criterion requires $e(H) > 8$. The case $e(H) = 10$ is again $K_5$ and has independence number one, while $e(H) = 9$ forces $H = K_5^-$, which has independence number two. If $\alpha = 3$, \eqref{eq:edgebound} gives $e(H) \leq 7 = 10 - \alpha$; for $\alpha = 4$ it gives $e(H) \leq 4 < 6$, and $\alpha = 5$ is impossible because $H$ has an edge. This proves the claim.
\end{proof}

\section{Consequences and remaining questions}

\begin{definition}
An \emph{ascending profile} with target $G$ of size $m$ is a sequence $\mathcal{H} = (H_1, \ldots, H_m)$ such that $e(H_i) = i$, $H_i$ is isomorphic to a subgraph of $H_{i+1}$ for $1 \leq i < m$, and $H_m \cong G$. A Ramsey chain $G_1, \ldots, G_m$ \emph{realises} $\mathcal{H}$ if $G_i \cong H_i$ for every $i$. Let $\TR(\mathcal{H})$ be the least integer $n$ such that every red--blue colouring of $K_n$ contains a Ramsey chain realising $\mathcal{H}$.
\end{definition}

\begin{remark}[Prescribed ascending profiles]
Every Ramsey chain realising an ascending profile $\mathcal{H}$ with target $G$ is, in particular, a Ramsey chain with target $G$. Hence $\TR(\mathcal{H}) \geq \TR(G)$. Therefore every graph supplied by Theorem~\ref{thm:planar}, Theorem~\ref{thm:dgrid}, or Theorem~\ref{thm:copies} satisfies
\[
\TR(\mathcal{H}) > \R(G)
\]
for every ascending profile $\mathcal{H}$ with target $G$.
\end{remark}

\begin{remark}[Minimal counterexamples]
The constructions above are asymptotic and do not determine the smallest counterexample, whether minimality is measured by order or by size. It remains natural to determine the least graph $G$ for which $\TR(G) > \R(G)$, and the least $t$ for which the square grid $\Gamma_t$ has this property. In the multiple-copy construction, quantitative bounds on the onset of the eventual linear regime, such as those of Buci\'c and Sudakov \cite{BS}, make the argument effective, although the resulting examples need not be small.
\end{remark}

\end{document}